\documentclass[reqno, 11pt]{amsart}
\usepackage{url}
\usepackage[hidelinks]{hyperref}
\usepackage{cleveref}
\usepackage{helvet}
\usepackage{ stmaryrd }
\usepackage{color,soul}
\usepackage{amsthm, thmtools}
\usepackage[style=numeric,maxbibnames=99]{biblatex}
\usepackage{caption} 
\usepackage{float}
\usepackage[margin=1in]{geometry}

\newtheorem{mainthm}{Theorem}

\newtheorem{cor}{Corollary}

\newtheorem{theorem}{Theorem}[section]
\newtheorem{lemma}[theorem]{Lemma}
\newtheorem{prop}[theorem]{Proposition}

\newtheorem{remark}[theorem]{Remark}

\theoremstyle{definition}

\newtheoremstyle{break}
  {\topsep}{\topsep}%
  {\itshape}{}%
  {\bfseries}{}%
  {\newline}{}%
\theoremstyle{break}

\numberwithin{equation}{section}

\usepackage{bbm}
\usepackage{euscript}
\usepackage{pb-diagram}
\usepackage{amsmath}

\usepackage{amsxtra}
\usepackage{amssymb}
\usepackage{pifont}
\usepackage{amsbsy}

\usepackage{graphicx}
\usepackage{epstopdf}
\usepackage{colortbl}
\usepackage{multirow}
\usepackage{hhline}

\newskip\aline \newskip\halfaline
\newcommand\bD{{\mathbb D}}

\newcommand{\bN}{{{\mathbb N}}}

\newcommand{\bP}{{{\mathbb P}}}

\newcommand\bR{{\mathbb R}}
\newcommand\bS{{\mathbb S}}

\DeclareMathOperator{\diam}{diam}

\DeclareMathOperator{\vol}{vol}

\newcommand{\eps}{{\epsilon}}

\def\inpr{\mathbin{\hbox to 6pt{\vrule height0.4pt width5pt depth0pt \kern-.4pt \vrule height6pt width0.4pt depth0pt\hss}}}

\newcommand{\hemi}{\bS^n_+}
\newcommand{\bhemi}{\partial\bS^n_+}

\def\XXint#1#2#3{{\setbox0=\hbox{$#1{#2#3}{\int}$ }
\vcenter{\hbox{$#2#3$ }}\kern-.59\wd0}}

\title{New Counterexamples to Min-Oo's Conjecture via Tunnels}

\author{Paul Sweeney Jr.}
\address{Paul Sweeney Jr.: Universit\`a di Trento, Dipartimento di Matematica, via Sommarive 14, 38123 Povo di Trento, Italy}
\email{paul(dot)sweeneyjr(at)unitn(dot)it}

\begin{document}
\begin{abstract}
Min-Oo's Conjecture is a positive curvature version of the positive mass theorem. Brendle, Marques, and Neves produced a perturbative counterexample to this conjecture. In 2021, Carlotto asked if it is possible to develop a novel gluing method in the setting of Min-Oo's Conjecture and in doing so produce new counterexamples. Here we build upon the perturbative counterexamples of Brendle--Marques--Neves in order to construct counterexamples that make advances on the theme expressed in Carlotto's question. These new counterexamples are non-perturbative in nature; moreover, we also produce examples with more complicated topology. Our main tool is a quantitative version of Gromov--Lawson Schoen--Yau surgery.
\end{abstract}
\maketitle
\section{Introduction} 
The interaction of positive scalar curvature and stable minimal surfaces has been a fruitful investigation in Riemannian geometry. The positive mass theorem is a landmark result in this direction. It says: if an asymptotically flat Riemannian manifold has nonnegative scalar curvature, then the mass of the manifold is nonnegative. Moreover, the mass is equal to zero if and only if the manifold is isometric to Euclidean space. The positive mass theorem was first proven for dimensions less than or equal to seven by Schoen and Yau \cite{SY} via minimal surface techniques. Using spinors, Witten \cite{W} proved the positive mass theorem in all dimensions for spin manifolds. Over the years, there have been new proofs in dimension three of the positive mass theorem using varying techniques. In 2001, Huisken and Ilmanen \cite{HI} proved the positive mass theorem by studying inverse mean curvature flow. More recently, Li \cite{L} used Ricci flow to prove the positive mass theorem. Moreover, there are several recent proofs of the positive mass theorem: Bray, Kazaras, Khuri, and Stern \cite{BKKS} via levels sets of harmonic functions, Miao \cite{Mi} via positive harmonic functions and capacity of sets, and Agostiniani, Mazzieri, and Oronzio \cite{AMO} via the Green's function.

Other related rigidity phenomena, involving positive scalar curvature and minimal surfaces, have been studied by Carlotto, Chodosh, and Eichmair \cite{CCE}. For example, the only asymptotically flat three-dimensional Riemannian manifold with nonnegative scalar curvature that contains a complete non-compact embedded surface $S$ which is a (component of the) boundary of some properly embedded full-dimensional submanifold and is area-minimizing under compactly supported deformations is Euclidean $\bR^3$. Moreover, $S$ is a flat plane. Now, this result should be compared with a special case of a localized gluing construction of Carlotto and Schoen \cite{CS}. For a nice discussion of the topic, see the article of Chru\'{s}ciel \cite{Ch}.  The general construction allows for localized solutions to the vacuum Einstein constraint equations and, roughly, states that one can construct asymptotically flat initial data sets that have positive ADM mass but are trivial outside a cone of a given angle. In particular, one can construct a metric on $\bR^3$ that is asymptotically flat, scalar flat, and is Euclidean on $\bR^2\times (0,\infty)$, but is not the Euclidean metric on $\bR^3$.

One is naturally led to wonder: Can something analogous be said for negative or positive curvature? The answer for negative curvature is that there are parallel results. For spin manifolds, Min-Oo \cite{M2} gave the first characterization of hyperbolic space in the context of a positive mass theorem, which was later refined by Andersson and Dahl \cite{AD}. By combining Chru\'{s}ciel and Herzlich \cite{CH1} with Wang \cite{Wa1}, the positive mass theorem for $n$-dimensional asymptotically hyperbolic spin manifolds was established in the negative curvature setting. Without the spin assumption, Andersson, Cai, and Galloway \cite{ACG} proved a hyperbolic positive mass theorem for $n$-dimensional manifolds, $3\leq n\leq 7$, with an additional hypothesis on the mass aspect function. Also, Sakovich \cite{Sa} used the Jang equation to prove the hyperbolic positive mass theorem in dimension three. In \cite{CD1}, Chru\'{s}ciel and Delay proved a positive mass theorem without rigidity for asymptotically hyperbolic manifolds in any dimension via a gluing argument. Lastly, Huang, Jang, and Martin \cite{HJM} proved the rigidity case for the hyperbolic positive mass theorem. For a history and a systematic presentation of these results see \cite[Appendix C]{C}. Moreover, there is also a localized gluing construction in the negative curvature setting due to Chru\'sciel and Delay \cite{CD}.

The story in the positive curvature setting differs from the other two. The analogous statement for the positive mass theorem is known as Min-Oo's Conjecture \cite{M}. The conjecture states: if $g$ is a smooth metric on the hemisphere $\hemi$ such that the scalar curvature, $R_g$, satisfies $R_g\geq n(n-1)$, the induced metric on the boundary $\bhemi$ agrees with the standard unit round metric on $\bS^{n-1}$, and the boundary $\bhemi$ is totally geodesic with respect to $g$, then $g$ is isometric to the standard unit round metric on $\hemi$. Many special cases of Min-Oo's Conjecture are known to be true (\cite{BM}, \cite{HW1}, \cite{HW2}, \cite{HLS}).

However, in \cite{BMN}, Brendle, Marques, and Neves constructed counterexamples to Min-Oo's Conjecture using perturbative techniques. They performed two perturbations to the standard unit round metric on $\hemi$. First, they perturbed the metric so that scalar curvature is strictly larger than $n(n-1)$ and the mean curvature of the boundary is positive. The second perturbation made the boundary totally geodesic while preserving the scalar curvature lower bound. 

In light of the above gluing constructions for zero and negative curvature, Carlotto asked in \cite[Open Problem 3.16]{C}: Can one design a novel class of counterexamples to Min-Oo's Conjecture based on a gluing scheme? More specifically, can one remove a neighborhood of a point on the boundary of $\bS^n_+$ and then use a gluing scheme analogous to Corvino \cite{Co} or Corvino--Schoen \cite{CoS} (also see \cite{CS} \cite{CD2}) to produce counterexamples to Min-Oo's conjecture.

In \cite{CEM}, Corvino, Eichmair, and Miao produced different counterexamples to Min-Oo's conjecture. Their gluing scheme says that given two smooth compact Riemannian manifolds $(M_i,g_i)$, $i=1,2$, of constant scalar curvature $\kappa$ which contain two non-empty domains $U_i\subset M_i$ where $g_i$ is \emph{not V-static} then one can construct a smooth metric on $M_1\#M_2$ with constant scalar curvature $\kappa$. Using this gluing theorem, they glued 3-spheres near the boundary of the perturbative counterexample constructed by Brendle, Marques, and Neves in \cite[Theorem 7]{BMN}. The resulting manifold satisfies the hypotheses of Min-Oo's Conjecture but has arbitrarily large volume. Their gluing method (which is related to the gluing methods in \cite{CIP}, \cite{IMP}, \cite{IMP2}) has two parts. The first is a conformal deformation to construct a neck connecting the two manifolds. The second is a deformation out of the conformal class to preserve the initial metrics away from the gluing region. 

A different gluing scheme found in \cite{S}, which is a quantitative version of Gromov--Lawson Schoen--Yau tunnels (\cite{GL}, \cite{SY1}), can be used to build upon the Brendle--Marques--Neves counterexamples \cite{BMN} to construct new and more extreme counterexamples to Min-Oo's Conjecture. This construction allows for the removal of some of the assumptions required in \cite{CEM}. In particular, we show the following:

\begin{mainthm}\label{mmainA}
    Let $D>0$, $n\geq3$, and $(M^n,g)$ be a Riemannian manifold such that the scalar curvature, $R_g$, satisfies $R_g> n(n-1)$. Let $\partial \hemi = \Sigma$. Then $N=M\#\hemi$ admits a metric $\tilde{g}$ such that:
    \begin{itemize}
        \item The scalar curvature, $\tilde{g}$, satisfies $R_{\tilde{g}}> n(n-1)$ everywhere.
        \item The induced metric on $\Sigma$ agrees with the standard unit round metric on $\bS^{n-1}$.
        \item  $\Sigma$ is totally geodesic with respect to $\tilde{g}$.
        \item  $\diam_{\tilde{g}}(N)\geq D$.
    \end{itemize}
\end{mainthm}

\begin{remark}
    We note, as opposed to the construction in \cite{CEM}, the gluing scheme employed here does not need the existence of non-V-static metrics on the manifolds that we glue together and we do not need an open subset of constant scalar curvature to which to glue.
\end{remark}

The following corollaries of \Cref{mmainA} provide examples that make strides on the theme present in Carlotto's question by exploring the shape of manifolds that satisfy the geometric hypotheses of Min-Oo's Conjecture.
Specifically, they show that these counterexamples can look geometrically very different than the standard unit round hemisphere. The first corollary shows we can have a counterexample with an arbitrarily large diameter.

\begin{cor}\label{mcorD}
    For any $D>0$, there exists a metric $g$ on $\hemi$, $n\geq 3$, such that:
    \begin{itemize}
        \item The diameter of $M=(\hemi, g)$ satisfies $\diam_g(\hemi)\geq D$.
        \item The scalar curvature satisfies $R_{{g}}> n(n-1)$ everywhere.
        \item The induced metric on $\partial M$ agrees with the standard unit round metric on $\bS^{n-1}$.
        \item  $\partial M$ is totally geodesic with respect to ${g}$.
    \end{itemize}
\end{cor}

We also construct examples of arbitrarily large volumes.

\begin{cor}\label{mcorV}
    For any $V>0$, there exists a metric $g$ on $\hemi$, $n\geq 3$, such that:
    \begin{itemize}
        \item The volume of $M=(\hemi, g)$ satisfies $\vol_g(\hemi)\geq V$.
        \item The scalar curvature satisfies $R_{{g}}> n(n-1)$ everywhere.
        \item The induced metric on $\partial M$ agrees with the standard unit round metric on $\bS^{n-1}$.
        \item  $\partial M$ is totally geodesic with respect to ${g}$.
    \end{itemize}
\end{cor}

\begin{remark}
    In \cite[Remark 6.4]{CEM}, Corvino, Eichmair, and Miao suggest one should be able to construct similar examples using a gluing technique akin to the Gromov--Lawson tunnel construction \cite{GL}. \Cref{mcorV} rigorously proves this remark.
\end{remark}

\begin{remark}
    We note that for our construction we cannot glue near the boundary of the counterexample in \cite[Theorem 7]{BMN} as they do in \cite{CEM}; however, we can glue around any point where the scalar curvature is strictly larger than $n(n-1)$ in  \cite[Theorem 7]{BMN}. Moreover, we can glue around any point in the counterexample constructed in \cite[Corollary 6]{BMN}, in particular near the boundary.
\end{remark}

We want to emphasize that in \Cref{mmainA} the only restriction on $M$ is that it admits a metric with scalar curvature strictly larger than $n(n-1)$. In particular, any closed Riemannian manifold that admits positive scalar curvature can be connected to a counterexample of Min-Oo's Conjecture. Therefore, we obtain new counterexamples with non-trivial topology. For example:

\begin{cor}\label{mcorT}
   Let $p,q,n\in \bN$ such that $n\geq 3$ and $p+q=n$. Then there exists a metric $g$ on $M=\hemi\#\left(\bS^p \times \bS^q\right)$ such that:
   \begin{itemize}
        \item The scalar curvature satisfies $R_{{g}}> n(n-1)$ everywhere.
        \item The induced metric on $\partial M$ agrees with the standard unit round metric on $\bS^{n-1}$.
        \item  $\partial M$ is totally geodesic with respect to ${g}$.
    \end{itemize}
\end{cor}

Corvino, Eichmair, and Miao are able to obtain counterexamples to Min-Oo's Conjecture with a non-trivial fundamental group by gluing a counterexample of Min-Oo's Conjecture to $\bS^1\times \bS^{n-1}$ or $\bS^n/\Gamma$ where $\Gamma$ is a finite subgroup of $SO(n+1)$. \Cref{mcorT} highlights that many more topologies can be produced.

We observe that the result of \Cref{mcorT} can be viewed as a $0$-surgery on the interior of $\hemi \cup (\bS^p \times \bS^q)$. The original construction of Gromov--Lawson \cite{GL} works for surgeries in codimension greater than or equal to three. Therefore, one may wonder if the connected sum procedure in \cite{S} can be extended to surgeries in codimension greater than or equal to three. We prove an analogous statement for surgeries in codimension greater than or equal to three in \Cref{mgluing}, which may be of independent interest. These surgeries are performed on the interior of the manifold. As a result, we get the following statement:
 \begin{mainthm}\label{mmainsurg}
     Let $n\geq 3$ and $M^n$ be a manifold obtained from performing surgeries in codimension greater than or equal to three on the interior of $\bS^n_+$. Then there exists a metric $g$ on $M$ such that:
     \begin{itemize}
        \item The scalar curvature, ${g}$, satisfies $R_{{g}}> n(n-1)$ everywhere.
        \item The induced metric on $\partial M$ agrees with the standard unit round metric on $\bS^{n-1}$.
        \item  $\partial M$ is totally geodesic with respect to ${g}$.
    \end{itemize}
 \end{mainthm}

In \cite[Theorem 7]{BMN}, Brendle, Marques, and Neves construct a metric $g$ on the $n$-hemisphere, $n\geq 3$, with the following properties: the scalar curvature is at least $n(n-1)$ everywhere, there exists a point where the scalar curvature is strictly larger than $n(n-1)$, and the metric agrees with the standard unit round metric on $\hemi$ in a neighborhood of the boundary. They showed that this metric can be used to construct a metric on $\bR\bP^n$, $n\geq 3$, that has analogous properties as the one on the hemisphere. We would like to point out that if we view the $g$ as a metric on an $n$-ball we can produce analogous metrics on any lens space by making the appropriate identifications on the boundary.

\begin{mainthm}
    Let $n\geq 3$. Then for any $n$-dimensional lens space $L^n$ there exists a metric $g$ such that $(L^n,g)$ has the following properties:
    \begin{itemize}
        \item The scalar curvature satisfies $R_g\geq n(n-1)$.
        \item There exists a point $p\in L$ such that $R_g(p)>n(n-1).$
        \item The metric $g$ agrees with the standard unit round metric in a neighborhood of the equator in $L.$
    \end{itemize}
\end{mainthm}

Related to the work of Brendle, Marques, and Neves is the following rigidity result of Miao and Tam \cite{MiT} for hemispheres. Let $g_1^n$ denote the standard unit round metric on $\hemi$. Their theorem states: if $g$ is a metric on $\hemi$ such that the scalar curvature satisfies $R_g\geq R_{g_1^n}$, the mean curvature of the boundary satisfies $H_g\geq H_{g_1^n}$, $g=g_1^{n-1}$ on the boundary, the volume satisfies $V_g\geq V_{g_1^n}$, and $g$ is sufficiently close to $g_1^n$ in the $C^2$ norm, then $g$ is isometric to $g_1^n$. We note that the theorem of Miao and Tam is false without the perturbative hypothesis. The first example showing this was constructed in \cite{CEM}. We note that \Cref{mcorV} gives an alternative construction showing the need for the $C^2$-closeness in \cite{MiT}. 

Moreover, we construct the following example which should be compared with \Cref{mcorV} and the rigidity result of Miao and Tam.
\begin{mainthm}\label{mmainB}
  Let $n\geq 3$, $0<\eps<\frac{1}{1000}$, and $D>0$. Then there exists a metric $g$ on $\hemi$ with the following properties:
    \begin{itemize}
        \item The scalar curvature satisfies $R_g>n(n-1)$.
        \item The mean curvature on $\bhemi$ satisfies $H_g>0$.
        \item The induced metric on $\bhemi$ is the standard unit round metric on $\bS^{n-1}$.
        \item  The volume satisfies 
        \[
        \frac{1}{2}\omega_n\leq 
        \vol_g(\hemi) \leq \frac{1}{2}\omega_n +\eps,
        \]
        where $\omega_n$ is the volume of the standard unit round $n$-sphere.
        \item The diameter satisfies $\diam_g(\hemi) >D$.
    \end{itemize}
\end{mainthm}

By relaxing the condition on the curvature on the boundary from totally geodesic to $H_g>0$, we construct examples that are related to the result of Miao and Tam while keeping the volume arbitrarily close to the volume of the standard unit round $\hemi$. Moreover, from the proof of \Cref{mmainB} one can see outside a set of arbitrarily small volume the metric $g$ is a small perturbation of the standard unit round metric on $\hemi$.

\section{Acknowledgements}
The author would like to thank the anonymous referee for a multitude of insightful comments which greatly improved the content and style of the paper. The author would like to thank Marcus Khuri and Raanan Schul for their many helpful discussions and suggestions throughout the process of producing this result. The author would also like to thank Alessandro Carlotto for his interest in this result. This paper was partially supported by Simons Foundation International, LTD. This work was supported in part by NSF Grant DMS-2104229 and NSF Grant DMS-2154613.

\section{Gluing Manifolds}
In this section, we will prove \Cref{mmainA}, \Cref{mcorD}, \Cref{mcorV}, \Cref{mcorT}, \Cref{mmainsurg}, and \Cref{mmainB}. These proofs follow from the perturbative constructions in \cite{BMN}, the gluing construction in \cite{S}, and \Cref{mgluing}. First, let us recall some results. The first proposition says that if the scalar curvature of a manifold is bounded below, then one can attach a tunnel and only decrease the lower bound by an arbitrarily small amount while maintaining control of the volume \emph{and} diameter.

\begin{prop}[{\cite[Proposition 4.2]{S}} Constructing Tunnels]\label{mpropT}
Let $(M_i^n,g_i)$, $n\geq 3$, $i=1,2$, be Riemannian manifolds with scalar curvature $R_{g_i}$. Let $\kappa\in \bR$, $d\geq 0$, $j\geq 1$, $p_i$ be a point in the interior of $M_i$, and $\delta\in \left(0, \frac{1}{10}\min\{\mathrm{inj_1}, \mathrm{inj_2}\}\right)$, where $\mathrm{inj}_i$ is the injectivity radius of $M_i$ at $p_i$. Assume $R_{g_i}\geq\kappa$ on $B_1=B_{g_1}(p_1,2\delta)\subseteq M_1$ and $B_2=B_{g_2}(p_2,2\delta) \subseteq M_2$. Then there exists a complete Riemannian metric $\bar{g}$ on the smooth manifold $P^n=M_1 \# M_2$, which is obtained by removing $B_i$ from $M_i$ and gluing in a cylindrical region $(T_\delta,g_\delta)$ diffeomorphic to $\bS^{n-1}\times[0,1]$, satisfying the following:
\begin{enumerate}
    \item the metrics $g_i$ agree with $\bar{g}$ on $M_i\setminus B_i$;
    \item there exists constant $C>0$ independent of $\delta$, $d$, and $j$ such that
\begin{equation}\label{e: prop3}
    d<\diam{({T_\delta})}<C\delta +d \quad \text{and} \quad \vol{({T_\delta})}<C(\delta^n+d\delta^{n-1});
\end{equation}

\item $T_\delta$ has scalar curvature $R_{g_\delta}>\kappa -\frac{1}{j}$.
\end{enumerate}
\end{prop}

The following is the statement of the counterexample of Min-Oo's Conjecture constructed by Brendle, Marques, and Neves.

\begin{theorem}[\cite{BMN}, Corollary 6] \label{mocounter}
    Let $n\geq 3$. There exists a Riemannian metric $g$ on the hemisphere $\hemi$ such that:
    \begin{enumerate}
        \item The scalar curvature, $R_g$, satisfies $R_g>n(n-1)$.
        \item The boundary $\partial \bS^n_+$ is isometric to the standard unit round $(n-1)$-sphere.
        \item The boundary $\bhemi$ is totally geodesic with respect to $g$.
    \end{enumerate}
\end{theorem}
Now we are ready to prove our main result.
\begin{proof}[Proof of \Cref{mmainA}]
   Let $n\geq 3$. Let $(M^n,g)$ be a Riemannian manifold such that the scalar curvature, $R_g$, satisfies $R_g>n(n-1)$. Let $(\hemi, \bar{g})$ be the Riemannian manifold from \Cref{mocounter}. Choose $d\geq D$ and $j$ large enough in \Cref{mpropT} so that the constructed metric on $M\# \hemi$ satisfies the conditions in \Cref{mmainA}.
\end{proof}

\begin{proof}[Proof of \Cref{mcorD}]
    Let $n\geq 3$. In \Cref{mmainA}, take $(M^n,g)$ to be $\left(\bS^n,g^n_\frac{1}{2}\right)$ where $g^n_\frac{1}{2}$ is the round metric on $\bS^n$ of radius $\frac{1}{2}$.
\end{proof}

\begin{proof}[Proof of \Cref{mcorV}]
    First, we will construct for any $n\geq 3$ a closed Riemannian manifold $(M^n,g)$ such that the scalar curvature is strictly larger than $n(n-1)$ and the volume is larger than $\frac{m}{4}\omega_n$, where $\omega_n$ is the volume of $\bS^n$ with its standard unit round metric.
    
    In particular, let $(N,h)$ be the round $n$-sphere of radius $\left(\frac{1}{2}\right)^\frac{1}{n}$ and note $\vol_h(N)= \frac{1}{2} \omega_n$ and $R_h > n(n-1)$. Let $M$ be the manifold obtained from $N$ by attaching (via a connected sum) $m$ round $n$-spheres of radius $\left(\frac{1}{2}\right)^\frac{1}{n}$ around the the equator of $N$. By \Cref{mpropT}, there exists a complete Riemannian metric $g$ on $M$ such that $R_g > n(n-1)$. Observe that $\vol_g(M) > m\times \left(\frac{1}{2} \vol_h(N)\right) = \frac{m}{4}\omega_n$. Choose $m$ large enough so that $\vol_g(M) > V$. 
    
    Now by \Cref{mmainA}, we obtain a metric on $\hemi$ satisfying the conditions of \Cref{mcorV}.
\end{proof}

\begin{proof}[Proof of \Cref{mcorT}]
   Let $n,p,q\geq 1$ and $p+q=n\geq 3$. In \Cref{mmainA}, take $(M^n,g)$ to be $(\bS^p\times \bS^q, h)$ where $h$ is a metric on $\bS^p\times \bS^q$ of scalar curvature strictly larger than $n(n-1)$. In particular, one can choose $h=\frac{1}{2(n(n-1))}\left(g^p_1+g^q_1\right)$ where we recall $g^m_r$ is the standard round metric on $\bS^m$ of radius $r$.
\end{proof}

\begin{proof}[Proof of \Cref{mmainsurg}]
    Let $n\geq 3$ and take $(M^n,g)$ in \Cref{mpropS} to be the manifold from \Cref{mocounter}.
\end{proof}

\begin{proof}[Proof of \Cref{mmainB}]
   Let $n\geq 3$. Recall that $g_1^n$ is the standard unit round metric on $\bS^n$. By \cite[Corollary 15]{BMN} there exists a smooth function $u:\hemi\to\bR$ and a smooth vector field $X$ on $\hemi$ such that for all $t$ small enough we have that the metric 
    \[
    g(t)=g_1^n+t\mathcal{L}_Xg_1^n + \frac{1}{2(n-1)}t^2ug_1^n
    \]
    has scalar curvature strictly greater than $n(n-1)$, the mean curvature of the boundary is strictly positive, and $g(t)=g_1^{n-1}$ on $\bhemi$.
    
     Fix $0<\eps<\frac{1}{1000}$ and $D>0$. We note that this implies $10^n\omega_n\eps^n<\frac{\eps}{3}$. Now choose $t_0$ small enough so that
     \[
     |\vol_{g_1^n}(\hemi)-\vol_{g(t_0)}(\hemi)|<\omega_n\eps^n,
     \]
     where $\omega_n$ is the volume of the $n$-sphere, the scalar curvature of $g(t_0)$ is strictly larger than $n(n-1)$, and the mean curvature of the boundary is strictly positive.
     Using $g(t_0)$ and \Cref{mpropT} we can construct a metric on $\hemi$ that satisfies the conditions of \Cref{mmainB}.  
     
     In particular, in the setting of \Cref{mpropT}, we consider $M_1=(\hemi, g(t_0))$ and $M_2=(\bS^n, g^n_{10\eps})$ and $p_i$ be in the interior of $M_i$, $i=1,2$. Choose $\delta>0$ and $d>D$ in \Cref{mpropT} such that $2\delta<\eps$, $C\left(\delta^n+d\delta^{n-1}\right)<\frac{\eps}{3}$, and $\vol_{g(t_0)}\left(B\left(p_1,2\delta\right)\right)<\omega_n\eps^n$. Also, choose $j$ large enough so that $\min\left\{R_{g^n_{10\eps}}, R_{g(t_0)} \right\} -\frac{1}{j}>n(n-1).$  Let $B_1=B_{g(t_0)}(p_1,2\delta)\subset M_1$ and $B_2=B_{g^n_{10\eps}}(p_2,2\delta) \subset M_2$.  
     
     Thus, we obtain a metric $g$ on $\hemi\#\bS^n$ which has scalar curvature strictly larger than $n(n-1)$, mean curvature of the boundary strictly positive, and $\diam_g\left(\hemi\#\bS^n\right)>D$. Now we will estimate $\vol_g(\hemi\#\bS^n)$. 
     
      \noindent Note 
     \begin{align*}
         \vol_g(\hemi\#\bS^n) &= \vol_{g(t_0)}(\hemi\setminus B_1) + \vol(T) + \vol_{g^n_{10\eps}}(\bS^n\setminus B_2).
     \end{align*}
     and
     \[
     \left|\vol_{g_1^n}(\hemi)-\vol_{g(t_0)}(\hemi\setminus B_1)\right| <2\omega_n\eps^n.
     \]
    Thus,
     \begin{align*}
         \vol_{g_1^n}(\hemi) &\leq \vol_{g(t_0)}(\hemi\setminus B_1) + 2\omega_n\eps^n 
         \\&\leq \vol_{g(t_0)}(\hemi\setminus B_1) + \vol(T) + (10^n-1)\omega_n\eps^n 
         \\& \leq \vol_g(\hemi\#\bS^n)
         \\&= \vol_{g(t_0)}(\hemi\setminus B_1) + \vol(T) + \vol_{g^n_{10\eps}}(\bS^n\setminus B_2) 
         \\&\leq \vol_{g_1^n}(\hemi) + 2\omega_n\eps^n+C(\delta^n+d\delta^{n-1}) +\omega_n\left(10\eps\right)^n,
     \end{align*}
     where in final inequality we use $\mathrm{(ii)}$ from \Cref{mpropT} to estimate $\vol(T)$. Therefore, by our choice of $\eps$ and $\delta$ we see that
     \[
     \frac{1}{2}\omega_n\leq \vol_g(\hemi\#\bS^n) \leq  \frac{1}{2}\omega_n+\eps.
     \]
\end{proof}

\appendix 
\section{\\Surgery Construction with Controlled Scalar Curvature}\label{mgluing}
In this appendix, we prove a technical proposition, which is a Gromov--Lawson type surgery construction in codimension three or larger with a quantitative lower bound on the scalar curvature. As expressed by Gromov--Lawson in \cite{GL} the higher surgery result is very similar to connected sum construction. Therefore, the proof will combine both the improved connected sum construction in \cite{S} and the proof in \cite{GL} (cf. \cite{Wa}). This result may be of independent interest.

\begin{prop}\label{mpropS}
Let $\kappa \in \bR$ and $j>0$. Let $(M^n,g)$, $n\geq 3$, be a Riemannian manifold with scalar curvature satisfying $R_g\geq \kappa$. Let $p+q=n$, $p\geq 1$, and $q\geq 3$. Let $S^p\subseteq M$ be an embedded $p$-sphere in the interior of $M$ with a trivial normal bundle. Let $\delta\in\left(0,\frac{1}{10}\mathrm{rad}_{S^p}\right)$ where $\mathrm{rad}_{S^p}$ is the normal injectivity radius of $S^p$. Let $T(\delta)$ (which is diffeomorphic to $\bS^p\times \bD^q$) be the $\delta$-neighborhood of $S^p$. Then there exists a Riemannian metric $\bar{g}$ on $N^n$, which is the smooth manifold obtained by performing a $p$-surgery on $M$, i.e.
$N=M\setminus T(\delta) \sqcup \left(\bD^{p+1}\times \bS^{q-1}\right),$
satisfying the following:
\begin{enumerate}
    \item the scalar curvature of $N$ satisfies $R_{\tilde{g}} > \kappa -\frac{1}{j}$; \label{1}
    \item  $\bar{g}=g$ on $M\setminus T(\delta);$\label{2}
    \item there is a continuous increasing function $\xi:[0,\infty]\to[0,\infty]$ independent of $j$ and $\delta$ where $\xi(x)\xrightarrow{x\to 0} 0$  such that 
    \[
   \vol_g\left(M\right) - \xi(\delta) \leq \vol_{\bar{g}} \left(N\right) \leq \vol_g\left(M\right) + \xi(\delta).\label{3}
    \]
\end{enumerate}
\end{prop}

To perform the surgery, we will construct a Riemannian metric on $\bD^{p+1}\times \bS^{q-1}$ and attach it to $M\setminus T$. The desired Riemannian metric will be the induced metric on a codimension one submanifold. The codimension one submanifold will be specified by a particular curve from which the metric will inherit the desired properties.

\subsection{A Submanifold defined by a curve}\label{subsub}
Let $(M^n,g)$ be a Riemannian manifold with scalar curvature $R_g\geq \kappa$. Let $p+q=n$, $p\geq 1$, and $q\geq 3$. Let $S^p$ be an embedded $p$-sphere in the interior of $M$ with a trivial normal bundle $NS$. By choosing global orthonormal sections $\nu_1,\ldots,\nu_q$, we specify a diffeomorphism $i:S^p\times \bR^q\to NS$. Now define $r:S^p \times\bR^q\to \bR$, $r(y,x)=||x||$, where $||\cdot||$ is the Euclidean norm on $\bR^q$. Let $S^p\times D^q(s) = \{(y,x)\in S^p\times \bR^q:||x||\leq s\}$. Choose $\delta\in\left(0,\frac{1}{10}\mathrm{rad}_{S^p}\right)$. Therefore, $\mathrm{exp}\circ i\big|_{S^p\times D^q(\delta)}$ is an embedding, where $\mathrm{exp}$ is the exponential map with respect to $g$. Denote the image of this map by $T(\delta)$ and the coordinates $(y,x)$ will be used to denote points on $T(\delta)$. Therefore, $r$ is the distance function to $S^p\times \{0\}$ in $T(\delta)$. Also, curves of the form $\{y\}\times \ell$ when $\ell$ is a geodesic ray in $D^q(\delta)$ emanating from the origin are geodesics in $T(\delta)$.

Let $\gamma$ be a smooth curve in $\bR^2$ and consider the submanifold 
\[
\Sigma = \{(t,y,x)\in \bR\times T(\delta):(t,r(y,x))\in \gamma\}
\]
of $(\bR\times T(\delta), dt^2+g)$. Endow $\Sigma$ with the induced Riemannian metric $g_\gamma$. For brevity, we will suppress the $\delta$ and refer to $T(\delta)$ as $T$ from here on out.
Now we want to calculate the scalar curvature of $\Sigma$. To do so we will need the following lemma from \cite{GL} (cf. \cite{D}). First, we define some notation. If $B(p,\delta)$ is a geodesic ball in a Riemannian manifold $(M,g)$, then in polar coordinates $(r,\boldsymbol{\theta})$ about $p$ the metric $g$ in $B(p,\delta)$ takes the form $dr^2+g_r$. Let $0<\eps<\delta$ and $S_\eps^{n-1} = \{(r,\boldsymbol{\theta})\in B(p,\delta): r=\eps\}$ be the geodesic spheres of radius $\eps$ in $B(p,\delta)$; moreover, the induced metric on $S_\eps^{n-1}$ is $g_\eps$.

\begin{lemma}[{\cite[Lemma 1]{GL}}]\label{mprincurv}
Let $B(p,\delta)$ be a geodesic ball in a Riemannian manifold $(M^n,g)$. The principal curvatures of $S^{n-1}_\eps$ in $B(p,\delta)$ are each of the form $\frac{1}{-\eps} + O(\eps)$ for small $\eps$. Then, as $\eps \to 0$, $\frac{1}{\eps^2}g_\eps \to \frac{1}{\eps^2}g^{n-1}_\eps =  g_1^{n-1}$ in the $C^2$-topology.
\end{lemma}
\begin{remark}
   The above lemma is written with the convention that the second fundamental form is $A(X,Y)=g(\nabla_X Y,N)$ and $N$ is the outward pointing normal.
\end{remark}

We note that submanifolds of the form $S=\bR\times (\{y\}\times \ell)$ of $\bR\times T$ are totally geodesic. To calculate the scalar curvature of $\Sigma$, fix $w \in \Sigma \cap S$. Let $e_1,\dots,e_n$ be an orthonormal basis of of $T_w({\Sigma})$ where $e_1$ is tangent to $\gamma$ and $e_2,\ldots, e_{q}$ is tangent to $\partial D^q$.
 Note that the for points in ${\Sigma}\cap S$ the normal $\nu$ to ${\Sigma}$ in $\bR\times T$ is the same as the normal to $\gamma$ in $S$. 

From the Gauss equations:

\[
R_{dt^2+g}(X,Y,Z,U)=R_{g_\gamma}(X,Y,Z,U) - A(X,U)A(Y,Z) + A(X,Z)A(Y,U)
\]
we see 
\[
\left(K_{g_\gamma}\right)_{ij}=\left(K_{dt^2+g}\right)_{ij} +\lambda_i\lambda_j.
\]
where $\lambda_i$ are principal curvatures corresponding to $e_i$ and $\left(K_{g_\gamma}\right)_{ij}$ and $\left(K_{dt^2+g}\right)_{ij}$ are the respective sectional curvatures. We note that $\lambda_1 = k$ where $k$ is the geodesic curvature of $\gamma$. For $i=2,\ldots,q$ we see by \Cref{mprincurv}
\[
\begin{split}
\lambda_i &= \langle\nabla_{e_{i}} e_{i}, \nu\rangle = \langle\nabla_{e_{i}}  e_{i}, \cos \theta \partial_t+\sin\theta \partial_r\rangle
=\cos \theta \langle\nabla_{e_{i}} e_{i}, \partial_t \rangle+\sin \theta \langle\nabla_{e_{i}}  e_{i}, \partial_r\rangle\\
&= \sin \theta \langle\nabla_{e_{i}}  e_{i},\partial_r\rangle
=\left(\frac{1}{-r} + O(r) \right) \sin \theta,
\end{split}
\]
where $\theta$ is the angle that between $\nu$ and the $t$-axis. And for $(q+1),\ldots, n$ we have
\[
\begin{split}
\lambda_i &= \langle\nabla_{e_{i}}  e_{i},\nu,\rangle = \langle\nabla_{e_{i}}  e_{i}, \cos \theta \partial_t+\sin\theta \partial_r\rangle
=\cos \theta \langle\nabla_{e_{i}}  e_{i},\partial_t\rangle+\sin \theta \langle\nabla_{e_{i}}  e_{i}, \partial_r \rangle\\
&= \sin \theta \langle\nabla_{e_{i}}  e_{i}, \partial_r\rangle
=O(1) \sin \theta.
\end{split}
\]
Now note that
\[
\left(K_{dt^2+g}\right)_{1j}=R_{dt^2+g}(e_j,e_1,e_1,e_j)=R_{dt^2+g}(e_j,\cos \theta \partial_r,\cos \theta \partial_r,e_j)=\cos^2\theta \left(K_g\right)_{\partial_r j}
\]
 and for $i\neq 1$ and $j\neq 1$
\[
\left(K_{dt^2+g}\right)_{ij}=R_{dt^2+g}(e_j,e_i,e_i,e_j)=\left(K_g\right)_{i j}.
\]
We see that
\[
R_{{g_\gamma}} = \sum_{i\neq j} \left(K_{{g_\gamma}}\right)_{ij}  = 2\sum_{j\geq 2}\left(\cos^2\theta \left(K_g\right)_{\partial_r, j} + k\lambda_j\right) + \sum_{\substack{i\neq j \\ i,j\geq 2 }}\left(  \left(K_g\right)_{ij} + \lambda_i\lambda_j\right).
\]
Therefore,
\begin{equation}\label{mscal}
    \begin{split}
        R_{{g_\gamma}} &= R_g - 2\mathrm{Ric}(\partial_r,\partial_r)\sin^2\theta 
        + (q-1)(q-2)\left(\frac{1}{r^2} + O(1)\right) \sin^2\theta
         \\&\qquad- (q-1)\left(\frac{1}{r} + O(r)\right) k \sin\theta,
    \end{split}
\end{equation}
where the terms involving sectional curvature combine to the first two terms, the $\lambda_i\lambda_j$ terms combine to give the second term, and the $k\lambda_j$ terms combine to give the last term.

Next, we will construct a curve $\gamma$ in $\bR^2$.
If we construct $\gamma:[0,L]\to \bR^2$, $\gamma(s)=(t(s),r(s))$, such that near $s=0$ we have that $\gamma$ is a vertical line segment on the $r$-axis, then $\Sigma$ will smoothly attach to $M\setminus T$. Moreover, we will construct $\gamma$ such that near $s=L$ we have that $\gamma$ is a horizontal line segment parallel to the $t$-axis. Using \eqref{mscal}, in place of \cite[(4.1)]{S}, during the construction of the curve in \cite[Section 4]{S} results in our desired curve.  We record the construction of $\gamma$ in the following lemma:
\begin{lemma}\label{mcurve}
   Let $S^p\subseteq M$ and $\Sigma$ be as above. Let $\delta\in\left(0, \frac{1}{10}\mathrm{rad}_{S^p}\right)$ and $j>0$. Then there exists a smooth curve with arclength parameterization $\gamma:[0, L]\to \bR^2$, $\gamma(s)=(t(s),r(s))$, such that $\mathrm{length}(\gamma)\leq C\delta$ and $R_{g_\gamma}>\kappa - \frac{1}{j}$, where $C$ is independent of $j$ and $\delta$. Moreover, near $s=0$, we have that $\gamma$ is a vertical line segment on the r-axis and $s=L$ we have that $\gamma$ is a horizontal line segment parallel to the t-axis.
\end{lemma}

Therefore, we have now constructed a Riemannian metric, $g_\gamma$, on $\Sigma$ which smoothly attaches to $M\setminus T$. Let $P$ be the Riemannian manifold defined as $P=\left(M\setminus T\right) \sqcup \Sigma$ and with a slight abuse of notation we will let $g_\gamma$ be the metric on all of $P$. Note that by construction $P$ is a smooth manifold with boundary, which satisfies $R_{g_\gamma} > \kappa -\frac{1}{j}$.

\begin{remark}\label{r:0surgery}
    Up until now the same argument would work for the $p=0$ case (i.e., \Cref{mpropT}). Further details will be given in Subsection \ref{0surgery}. For now we will focus on the cases $p\geq1$.
\end{remark}

Now we must make a technical point that will become apparent later. Fix $\psi:[0,1]\to\bR$ a smooth increasing function satisfying
 \[
 \psi(s)=
 \begin{cases}
     0 & s\in\left[0,\frac{1}{4}\right]\\
     1 & s\in \left[\frac{3}{4},1\right].
 \end{cases}
 \]
Define $\Lambda_n= 2(n-1)|\psi''(s)|+(n-1)(n-2)|\psi'(s)|$. 
By taking $\delta$ smaller, if necessary, we may assume that the scalar curvature of $\partial P$ with its induced metric $g_{\partial P}$ satisfies
 \[
 R_{g_{\partial P}}>\kappa+\Lambda_n-\frac{1}{j}.
 \]

Now we want to find a homotopy, $\{h_s\}$, through metrics with scalar curvature strictly greater than $\kappa+\Lambda_n-\frac{1}{j}$, from the induced metric on $\partial P$ to $g^p_a+g^{q-1}_b$ on $\bS^p \times \bS^{q-1}$ where $g^m_\tau$ is the standard round metric on $\bS^m$ of radius $\tau$. Define $\bar{\delta}= \min\left\{\frac{1}{10\sqrt{|\Lambda_n+\kappa|}},\delta, 1\right\}$. By taking $a\leq1$ and $b<\bar{\delta}$ we ensure scalar curvature of $g^p_a+g^{q-1}_b$ is strictly larger than $\kappa+\Lambda_n-\frac{1}{j}$. Then one can construct a metric $ds^2+h_s$ on the collar, $C=[0,2]\times \bS^p \times \bS^{q-1}$, such that at $h_0$ is the induced metric on $\partial P$ and at $h_1$ is a product metric $g^p_a+g^{q-1}_b$ on $\bS^p \times \bS^{q-1}$. Finally, we can glue in $\left(E=\bD^{p+1}\times\bS^{q-1}, g_E\right)$ where $g_E$ is the product metric of a round metric on the hemisphere and a round metric on the sphere to complete the surgery. We will handle these steps in the next lemmas.

A version of the following lemma was first proved in \cite{GL} (cf. \cite{Wa}) where one wanted to ensure the Riemannian metrics in the homotopy have positive scalar curvature. Here we upgrade this lemma so that the homotopy is through metrics with scalar curvature strictly larger than $\kappa+\Lambda_n-\frac{1}{j}$.

\begin{lemma}\label{mhomotopy}
     Let $g_{\partial P}$ be the induced metric on $\partial P$ from $g_\gamma$. Then there exists a homotopy through Riemannian metrics with scalar curvature strictly larger than $\kappa+\Lambda_n-\frac{1}{j}$ to the product metric $g^p_1+ g_\eta^{q-1}$ for some $\eta\in(0,1)$.
\end{lemma}
\begin{proof}
    We follow the proof of \cite{Wa}. Let $g_{\partial P}$ be the induced metric on $\partial P$ from $g_\gamma$, recall $r(L)<\delta$, and set $\eta=r(L)$. Let $\pi: NS^p\to S^p$ be the normal bundle to the embedded $S^p$ in $M$. The Levi-Civita connection on $M$, with respect to $g_\gamma$, gives rise to a normal connection on the total space of the normal bundle and so also a horizontal distribution $\mathcal{H}$ on the total space of $NS$. Equip the fibers of $NS$ with the metric $\hat{g}=h_\eta$, where $h_\eta$ is the Riemannian metric on the $q$-ball that arises from the upper hemisphere of a round $q$-sphere of radius $\eta$. Now consider the Riemannian submersion $\pi:(NS,\tilde{g})\to (\bS^p, \check{g})$ where $\check{g}=g_{\partial P}$ and $\tilde{g}$ is the unique submersion metric arising from $\hat{g}$, $\check{g}$, and $\mathcal{H}$. Now we want to restrict this metric $\tilde{g}$ to the 
    sphere bundle, $\bS^p\times \bS^{q-1}(\eta)$, and by a slight abuse of notation we will call this restricted metric $\tilde{g}$. We note that $\tilde{g}$ restricted to $\bS^{q-1}_\eta$ is $g^{q-1}_\eta$, i.e., the round metric on the $(q-1)$-sphere of radius $\eta$. 

    Now by \Cref{mprincurv} (cf. \cite[Lemma 3.9]{Wa}) we have that $g_{\partial P}$ converges to $\tilde{g}$ as $\eta \to 0$. Therefore for small $\eta$, there is a homotopy from $g_{\partial P}$ to $\tilde{g}$ through metrics with scalar curvature strictly greater than $\kappa+\Lambda_n-\frac{1}{j}$. Now we want to construct a homotopy from $\tilde{g}$ to the product of two round spheres. Viewing $\tilde{g}$ as a submersion metric and applying O'Neill's formula \cite[Chapter 9]{Bes} we have the following equation for scalar curvature:
    \[
     \tilde{R} = \check{R} \circ \pi + \hat{R} - |A|^2 -|T|^2 - |n|^2 -2\check{\delta}(n).
    \]
    $T$ is a tensor that measures the obstruction to the bundle having totally geodesic fibers. By construction, the fibers are totally geodesic so $T = 0$ and $n$ is the mean curvature vector and so vanishes when $T$ vanishes. $A$ is O'Neill's integrability tensor and it measures the integrability of the horizontal distribution, i.e., when $A$ vanishes the horizontal distribution is integrable. Therefore, we have
    \[
    \tilde{R} = \check{R} \circ \pi + \hat{R} - |A|^2
    \]
    We want to deform $\tilde{g}$ through Riemannian submersions to one with a base metric $g^p_1$ while keeping scalar curvature larger than $\kappa+\Lambda_n-\frac{1}{j}$. This can be done because the deformation occurs on a compact interval and we can shrink $\eta$ to make $\hat{R}$ arbitrarily large. We just need to ensure when we shrink $\eta$ that the $|A|$ term does not grow. This follows from the canonical variation formula \cite[Chapter 9]{Bes} which states that if we shrink the fiber metric by $t$ then the scalar curvature of the new submersion metric $\tilde{R}_t$ satisfies
    \[
     \tilde{R}_t = \check{R} \circ \pi + \frac{1}{t}\hat{R} - t|A|^2
    \]
    Finally, we can perform another linear homotopy through Riemannian submersions to the standard product metric $g^p_1+g^{q-1}_\eta$, i.e., where $|A|=0$. Again we can shrink $\eta$ if necessary in order to preserve the scalar curvature bound.
\end{proof}

Thus, we have constructed the desired homotopy $h_s$. Recall our goal is to construct a metric on, $ds^2+h_s$ the collar, $C=[0,2]\times\bS^p\times\bS^{q-1}$ and maintain the scalar curvature bound. We now construct the metric on the collar. First, define $C'=[0,1]\times\bS^p\times\bS^{q-1}$ and recall that $\bar{\delta} = \min\left\{\frac{1}{10\sqrt{|\Lambda_n+\kappa|}},\delta,1\right\}$.
\begin{lemma}\label{concord1}
 Let $\psi:[0,1]\to\bR$ be the function fixed above. Recall that $\psi$ is a smooth increasing function satisfying
 \[
 \psi(s)=
 \begin{cases}
     0 & s\in\left[0,\frac{1}{4}\right]\\
     1 & s\in \left[\frac{3}{4},1\right].
 \end{cases}
 \]
 Let $\beta\in(0,\bar{\delta})$ and define $\varphi(s)=1-(1-\beta^2)\psi(s)$.
Then the scalar curvature of $(C', ds^2+\varphi^2(s) h_0)$ satisfies
   \[
   R_{ds^2+\varphi^2(s)h_0}\geq  R_{h_0}-\Lambda_n > \kappa-\frac{1}{j}
   \]
\end{lemma}
\begin{proof}
    Note that $(C',ds^2+\varphi^2(s) h_0)$ is a warped and that $\beta^2\leq|\varphi|\leq 1$. Thus, by \cite[Proposition 1.13]{LeeBook}
    \begin{align*}
    R_{ds^2+\varphi^2(s) h_0} &= \frac{R_{h_0}-2(n-1)\varphi(s)\varphi''(s)-(n-1)(n-2)|\varphi'(s)|}{\varphi^2(s)}\\
    &\geq\frac{R_{h_0}-2(n-1)|\psi''(s)|-(n-1)(n-2)|\psi'(s)|}{\varphi^2(s)}\\
    &\geq R_{h_0}-2(n-1)|\psi''(s)|-(n-1)(n-2)|\psi'(s)|
    \end{align*}
\end{proof}
\begin{lemma}\label{concord2}
    There exists a $\beta\in(0,\bar{\delta})$ such that $(C',ds^2+\beta^2 h_s)$ has scalar curvature larger than $\kappa-\frac{1}{j}$.
\end{lemma}
\begin{proof}
  
   We note that we can view $ds^2+\beta^2h_s$ as a submersion metric where we are shrinking the fibers by $\beta^2$. Therefore, by the canonical variation formula \cite[Chapter 9]{Bes} we have
   \[
   R_{ds^2+\beta^2 h_s} = \frac{1}{\beta^2} R_{h_s} +\beta^2 |A|
   \]
   Now choose $0<\beta<\bar{\delta}$ such that 
    \[
    R_{ds^2+\beta^2 h_s}=\frac{1}{\beta^2}R_{h_s}+\beta^2 |A| >  \kappa -\frac{1}{j}
    \]
\end{proof}

We note that the metric from \Cref{concord1} smoothly attaches to $P$ and the metric from \Cref{concord2} smoothly attaches to $(C',ds^2+\phi^2(s)h_0)$ (i.e. the  metric from \Cref{concord1}). Thus, we have constructed the metric (which with a slight abuse of notation we will be denoted as $ds^2+h_s$) on the collar $C$ that transitions from the induced metric on $\partial P$ to  $g^p_a+g^{q-1}_b$, where $a=\beta$ and $b=\beta\eta$, such that scalar curvature is larger than $\kappa-\frac{1}{j}$ because our choice of $\beta$ was small enough. Finally, recall that now we can attach $E$. Thus, we have completed the construction of $(N,\bar{g})$ and verified \eqref{1} and \eqref{2}.

Finally, we can compute the volume estimate \eqref{3}. We note from the construction that $N$ consists of four disjoint pieces glued together: $M\setminus T$, $\Sigma$, $C$, and $E$.  

\begin{lemma}
    Consider $(N,\bar{g})$ as constructed above. Then there is a continuous increasing function $\xi:[0,\infty]\to [0,\infty]$ independent of $j$ and $\delta$ where $\xi(x)\xrightarrow{x\to 0} 0$ such that 
    \[
 \vol_g\left(M\right) - \xi(\delta) \leq \vol_{\bar{g}} \left(N\right) \leq \vol_g\left(M\right) + \xi(\delta).
    \]
\end{lemma}
\begin{proof}
     We note that
  \[
   \vol_g\left(M\right) -  \vol_g\left(T\right) \leq \vol_{\bar{g}} \left(N\right) \leq \vol_g\left(M\right) + \vol_{\bar{g}}\left(\Sigma\right) +\vol_{\bar{g}}\left(C\right) + \vol_{\bar{g}}\left(E\right)
  \]
  It suffices to show that to bound the volume of $\Sigma$, $C$, $E$, and $T$. Since the metric on $E=\bD^{p+1}\times \bS^{q+1}$ is the product metric the round metric of radius $a$ on the $(p+1)$-hemisphere with the round metric of radius $b$ on a $(q-1)$-sphere, we have that 
   \[
   \vol(E)\leq \omega_{p+1}\omega_{q-1} a^{p+1}b^{q-1} \leq \omega_{p+1}\omega_{q-1} \delta^{q-1},
   \]
   where $\omega_m$ is the volume of the unit round  $m$-sphere. 
   
   Let $\xi_i:[0,\infty]\to[0,\infty]$, $i=1\ldots 5$, be continuous increasing functions such that $\xi_i(x)\xrightarrow{x\to 0} 0$. The volume bound for $C$ will follow from applying the coarea formula. 
   
   First, recall that the homotopy $h_s$ has two parts $I_1=[0,s_0]$ and $I_2=[s_0,2]$. Let $I_1$ correspond to the homotopy from $g_{\partial P}$ to $\tilde{g}$ which comes from the $C^2$-convergence of $g_{\partial P}$ to $\tilde{g}$. Therefore, $|\vol_{h_s}-\vol_{\tilde{g}}|\leq \xi_1\left(\delta\right)$ for $s\in I_1$ where $\xi_1$ is independent of $j$ and $\delta$. Let $I_2$ correspond to the remaining part of the homotopy and recall for all $s\in I_2$, we have that $h_s$ are submersion metrics.
   
   It will suffice to estimate the volume for $h_s$ for $s\in I_2$ because $\tilde{g}=h_{s_0}$ and $\vol_{\tilde{g}}$ is sufficiently close to the volumes of $h_s$, $s\in I_1$. Now for $s\in I_2$, consider $\pi_s:\left(\bS^p\times\bS^{q-1}, h_s\right)\to \left(\bS^p, g_s\right)$ the submersion associated with $h_s$. Using the coarea formula, we then have
   \[
   \vol_{h_{s}}(\bS^{p}\times \bS^{q-1}) = \int_{\bS^p} \left(\int_{\pi_s^{-1}(y)}dV_{h_{s} |_{\pi_s^{-1}(y)}} (x)\right) dV_{g_s}(y) \leq\xi_2(\delta)\vol_{g_s}\left(\bS^p\right),
   \]
   where the last line follows by \Cref{mprincurv} and $\xi_2$ is independent of $j$ and $\delta$. By the compactness of $I_2$ we have that $\vol_{g_s}\left(\bS^p\right)$ is bounded by some constant $c_3$ independent of $s$; moreover, by the construction of the submersions and the homotopy, $c_3$ is independent of $\delta$ and $j$. Therefore, by the coarea formula
   \[
       \vol_{ds^2+h_s}(C)=\int_0^2 \vol_{h_{s}}(\bS^{p}\times \bS^{q-1}) ds\leq  \xi_3(\delta),
   \]
   where $\xi_3$ is independent of $j$ and $\delta$. As $T$ is a tubular neighborhood of $S^p\subset M$, it has a metric of the form $g_T=dr^2+g_{r}$ on $[0,\delta]\times \bS^p\times \bS^{q-1}$ and has volume less than $\xi_4(\delta)$, where $\xi_4$ is independent of $j$ and $\delta$. By \Cref{mprincurv}, \Cref{mcurve}, and the definition of $\Sigma$, the metric on $\Sigma$ is of the form $g_\gamma=ds^2+g_{r(s)}$ on $[0,\mathrm{length}(\gamma)]\times \bS^p\times \bS^{q-1}$ with $\mathrm{length}(\gamma)\leq C\delta$. Thus, we have $\vol_{g_\gamma}(\Sigma)\leq \xi_5(\delta)$, where $\xi_5$ is independent of $j$ and $\delta$. 
\end{proof}

\subsection{0-Surgery}\label{0surgery}
One may wonder if \Cref{mpropT} can be derived using the same proof as \Cref{mpropS}. The answer is yes with minor variations (for full details see \cite{D, S}). For the sake of completeness, we will quickly repeat these details. Recalling Subsection \ref{subsub}, we note when $p=0$ that $T(\delta)$ is a disjoint union of two embedded disks. If we replace $T(\delta)$ with a connected component $T_1(\delta)$ of $T(\delta)$ then we can follow the same argument up to \Cref{r:0surgery} (i.e., the construction of the metric $g_\gamma$ on $(M\setminus T)\sqcup \Sigma$). Specifically, using the curve $\gamma_1(s)=((t_1(s),r_1(s))$ provided by \Cref{mcurve}, we have constructed $\Sigma_1$ which is half of a tunnel. Analogously, for the other connected component $T_2(\delta)$ of $T(\delta)$, we construct a curve $\gamma_2(s)=((t_2(s),r_2(s))$ and $\Sigma_2$, which is another half of a tunnel. By making the same choices in the construction of $\gamma_1$ and $\gamma_2$, we can ensure that $r_1(L)=r_2(L)$. Now we will modify the metric at the end of $\Sigma_1$ so that it is a product metric of a round sphere and an interval (as is done in \cite{D}). Let $0<\alpha<\delta$ be small enough so that for $L-\alpha \leq s\leq L$, we have $\gamma_1(s)$ is a horizontal line segment parallel to the $t$-axis. Therefore, by the construction of $\gamma_1$, we have for $s\in [L-\alpha, L]$ that $t_1(s)=s$ and $r_1(s)\equiv c$ for some constant $c$. Let $a=t_1(L-\alpha)=L-\alpha$ and $b=t_1(L)=L$. We note that the induced metric on the end of the half-tunnel, $\{(t,y)\in \Sigma_1: a\leq t\leq b\}$, is $h_1=g_c+dt^2$, where  $g_c$ is the induced metric on the $(n-1)$-sphere $t^{-1}(\{b\})$. Let $h_1'=c^2g_{1}^{n-1}+dt^2$ and $\phi(t)=\psi\left(\frac{t-a}{\alpha}\right)$ where $\psi(u)$ is a smooth function on $[0,1]$ vanishing near zero, increasing to 1 at $u=\frac{3}{4}$ and equal to 1 for $u>\frac{3}{4}$. Define the metric $h$ for $t\in[a,b]$ as
$h(q,y)=g_c(q,y)+\phi(t)\left(c^2g_{1}^{n-1}-g_c\right)+dt^2.$
This metric transitions smoothly between $h_1$ and $h_1'$. Note
\[
h-h_1=\phi(t)\left(c^2g_{1}^{n-1}-g_c\right) = \phi(t)c^2\left(g_{1}^{n-1}-\frac{1}{c^2}g_c\right)
\]
and that the first and second derivatives of $\phi(t)$ are $O(\alpha^{-1})$ and $O(\alpha^{-2})$, respectively. So by \Cref{mprincurv}, we have that the second derivatives of $h-h_1$ are $O(\alpha^2)$. Therefore, for $\alpha$ small enough, the scalar curvature of $h$ is close to the scalar curvature of $h_1$ which, again by \Cref{mprincurv}, has scalar curvature larger than $\kappa-\frac{1}{j}$ for small enough $\alpha$. Therefore, we have changed the metric at the end of $\Sigma_1$ so that it looks like $c^2g_{1}^{n-1}+dt^2$. We can perform the analogous procedure to $\Sigma_2$. Therefore, we can immediately glue together the two half-tunnels and construct a tunnel while keeping the scalar curvature larger than $\kappa-\frac{1}{j}$. If we wish to make a long tunnel, then before gluing the two half-tunnels together insert a round cylinder $A_{c,d}=([0,d]\times \bS^{n-1}, dt^2+c^2g_1^{n-1})$. The estimates in \eqref{e: prop3} follow because $d<\diam(A_{c,d})<d+2\pi c^2$ and $\vol(A_{c,d})=d\omega_{n-1}c^{n-1}$ and outside of $A_{c,d}$ the length of $\gamma_i$, $i=1,2$, is $O(\delta)$ and each $t=constant$ cross-section is close to a round sphere of radius less than $\delta$.

\printbibliography

\end{document}